\documentclass[12pt,twoside,reqno]{amsart}
\usepackage[letterpaper,textwidth=6.9in,textheight=9.6in,centering]{geometry}
\usepackage{amsmath, amsfonts, amssymb}
\usepackage{mathrsfs}
\usepackage{latexsym}
\usepackage{mathtools}
\usepackage{exscale}
\usepackage{cases}
\usepackage{enumitem}
\usepackage{needspace}

\usepackage[pdfstartview=FitH,
            CJKbookmarks=true,
            bookmarksnumbered=true,
            bookmarksopen=true,
            hypertexnames=false,
            colorlinks=true,
            linkcolor=blue,
            anchorcolor=blue,
            citecolor=red,
            urlcolor=blue
            ]{hyperref}
\hypersetup{
    pdftitle={On the Reverse Projection Inequality},
    pdfauthor={},
    pdfsubject={The sharp three-dimensional inequality, its equality cases, and geometric high-dimensional counterexamples},
    pdfkeywords={projection body, reverse projection inequality, mixed volume, equality case, simplex section, cube projection, geometric perturbation}
}

\numberwithin{equation}{section}
\allowdisplaybreaks

\newtheorem{theorem}{Theorem}[section]

\newtheorem{lemma}[theorem]{Lemma}
\newtheorem{corollary}[theorem]{Corollary}

\newtheorem{proposition}[theorem]{Proposition}

\newtheorem{definition and theorem}[theorem]{Definition and Theorem}

\def\bl{\begin{lemma}}
\def\el{\end{lemma}}
\def\bc{\begin{corollary}}
\def\ec{\end{corollary}}
\def\bt{\begin{theorem}}
\def\et{\end{theorem}}

\def\bp{\begin{proposition}}
\def\ep{\end{proposition}}
\def\be{\begin{equation}}
\def\ee{\end{equation}}
\def\baa{\begin{align*}}
\def\eaa{\end{align*}}

\theoremstyle{definition}
\theoremstyle{remark}

\newcommand{\R}{\mathbb{R}}

\newcommand{\Sph}{\mathbb S}

\author[Y. Feng]{Yibin Feng}
\address{Yibin Feng, School of Mathematics and Statistics,
	Lanzhou University,
	Lanzhou, 730000, P.R. China.}
\email{fengyb@lzu.edu.cn}

\author[D. Zhang]{Ding Zhang}
\address{Ding Zhang, School of Mathematics and Statistics,
	Lanzhou University,
	Lanzhou, 730000, P.R. China.}
\email{zhding2023@lzu.edu.cn}

\begin{document}
\newcommand{\Kzero}{\mathcal{K}_o^n}
\newcommand{\Haus}{\mathcal{H}}
\newcommand{\pcap}{C}
\newcommand{\convlev}[2]{\langle #1\rangle_{#2}}
\newcommand{\gammaconst}{\gamma_{n,p}}
\newcommand{\avg}[1]{\left\langle #1\right\rangle_{\!\sigma}}

	\title[Integral Affine Surface
	Area-Variational Capacity Inequalities]
	{$L_p$-Integral Affine Surface Area-Variational Capacity Inequalities}
	
	
	
	\begin{abstract}
	 For $n\geq 2$ and $1<p<n$, Ludwig-Xiao-Zhang \cite{LXYZ} asked whether the $L_p$-integral affine
	surface area of every convex body containing the origin in its interior is bounded above by the ball-normalized multiple of \(L_p\)-variational capacity, and whether the analogous analytic inequality holds for
	$L_p$-convexified level sets.  We answer both questions negatively.
	\end{abstract}
	
	\date{}
	\maketitle

	\section{Introduction}
	
	The setting of this paper is the $n$-dimensional Euclidean space $\mathbb{R}^n$. A convex body in $\mathbb{R}^n$ is a compact convex set that has a non-empty interior. The unit sphere in $\mathbb{R}^n$ is denoted by $\mathbb{S}^{n-1}$. Denote by $\mathcal{K}_o^n$ the set of all convex bodies that contain the origin in their interiors. Let \(B^n\subset\mathbb{R}^n\) denote the Euclidean unit ball, and write \(\omega_n=V(B^n)\) for its volume. We use $\mathcal{H}^k$ to denote the $k$-dimensional Hausdorff measure and $\partial K$ the boundary of a convex body $K$. 
	
	A convex body \(K\) is uniquely determined by its support function
	\(h_K\colon\mathbb{R}^n\to[0,\infty)\), given by
	\[
	h_K(x)=\sup\{x\cdot y:y\in K\},
	\ \ x\in\mathbb{R}^n,
	\]
	where \(\cdot\) denotes the standard Euclidean inner product on
	\(\mathbb{R}^n\).
	
	For \(p\geq 1\), the Minkowski-Firey \(L_p\)-sum of two convex
	bodies \(K\) and \(L\), denoted by \(K+_pL\), is defined through its
	support function by
	\[
	h(K+_pL,\cdot)^p
	=
	h(K,\cdot)^p+h(L,\cdot)^p.
	\]

For \(K,L\in\mathcal{K}_o^n\), their \(L_p\)-mixed volume is defined,
in \cite{LutwakFirey}, by the following variational formula:
\[
V_p(K,L)
=
\frac{p}{n}
\lim_{\varepsilon\to 0^+}
\frac{
	V\bigl(K+_p\varepsilon^{1/p}L\bigr)-V(K)
}{
	\varepsilon
}.
\]

A compact set \(M\subset\mathbb{R}^n\) is called a star body with
respect to the origin if
\(
[0,x]\subset M
\ \text{for every }x\in M.
\)
If, in addition, \(\partial M\) is Lipschitz, then \(M\) is called a
Lipschitz star body. In this case, the unit outer normal
\(\nu_M(x)\) exists for \(\mathcal{H}^{n-1}\)-almost every
\(x\in\partial M\). The radial function of \(M\) is the map
\(
\rho_M\colon\mathbb{R}^n\setminus\{0\}\to\mathbb{R}, 
\)
 which is given by
\[
\rho_M(x)
=
\max\{\lambda\geq 0:\lambda x\in M\},
\ \ x\neq 0.
\]
For a Lipschitz star body \(M\), define its core by
\[
D_M
=
\left\{
tx:t>0,\ x\in\partial M,\
\left|x\cdot\nu_M(x)\right|>0
\right\}.
\]
The restriction of the radial function \(\rho_M\) to \(D_M\) is
locally Lipschitz continuous.

Let \(M\subset\mathbb{R}^n\) be a Lipschitz star body and
\(L\subset\mathbb{R}^n\) a convex body. Following
\cite[Section~5.2]{LXYZ}, define
\begin{equation*}\label{}
	V_p(M,L)
	=
	\frac{1}{n}
	\int_{\partial M}
	h_L\bigl(\nu_M(x)\bigr)^p
	\left|x\cdot\nu_M(x)\right|^{1-p}\,dx.
\end{equation*}
Although \(V_p(M,L)\) is well defined, it may be infinite.

Let \(M\subset\mathbb{R}^n\) be a Lipschitz star body. For
\(p\geq 1\), the \(L_p\)-integral affine surface area \(\Phi_p(M)\) is
defined by (see \cite[Section~5.2]{LXYZ})
\begin{equation*}
	\Phi_p(M)
	=
	\frac{(n\omega_n)^{\frac{n+p}{n}}}{\beta_{n,p}}
	\left(
	\int_{\mathbb{S}^{n-1}}
	V_p(M,\bar{u})^{-\frac{n}{p}}\,du
	\right)^{-\frac{p}{n}},
\end{equation*}
where the normalizing constant \(\beta_{n,p}\) is given by
\begin{equation*}
	\beta_{n,p}
	=
	V_p(B^n,\bar{u})
	=
	\frac{1}{2^p n}
	\int_{\mathbb{S}^{n-1}}
	|u\cdot v|^p\,dv
\end{equation*}
with $
\bar{u}=[-u/2,u/2]$.

The \(L_p\)-integral affine surface area is a
projection affine functional. Its case \(p=1\) reduces
to the classical integral affine surface area associated with the
polar projection body, whose origins lie in Minkowski's theory of
projection bodies; see \cite[Section~9.1]{Schneider}. The \(L_p\) framework of this theory originated with Firey's introduction of \(L_p\)-addition \cite{Firey}. Further, Lutwak \cite{LutwakFirey} developed this
framework by establishing the theory of \(L_p\)-mixed volumes and
posing the \(L_p\)-Minkowski problem. Based on these developments, Lutwak-Yang-Zhang \cite{LYZ} introduced the notion of \(L_p\) projection
bodies and established the corresponding affine isoperimetric
inequalities, providing a natural geometric interpretation of the \(L_p\)-integral affine surface area.

The functional counterpart of the theory of projection bodies began with
Zhang's sharp affine Sobolev inequality \cite{ZhangAffine}. Subsequently, Lutwak-Yang-Zhang \cite{LYZSobolev} extended Zhang's result to the sharp \(L_p\) affine Sobolev inequality. Further developments include affine
Moser-Trudinger and Morrey-Sobolev inequalities
\cite{CianchiLutwakYangZhang}, asymmetric affine \(L_p\) Sobolev
inequalities \cite{HaberlSchuster}, and the asymmetric affine
P\'{o}lya-Szeg\"{o} principle \cite{HaberlSchusterXiao}. Ludwig-Xiao-Zhang \cite{LXYZ} employed the \(L_p\) convexification
of level sets to establish sharp convex Lorentz-Sobolev inequalities
involving the \(L_p\)-integral affine surface area.

Let \(E\subset\mathbb{R}^n\) be compact. For \(p\geq 1\), the
\(L_p\)-variational capacity of \(E\), denoted by \(C_p(E)\), is defined by
\[
C_p(E)
:=
\inf\left\{
\|\nabla f\|_p^p:
f\in C_0^\infty(\mathbb{R}^n),\
f\geq\mathbf{1}_E
\right\},
\]
where \(\mathbf{1}_E\) denotes the indicator function of \(E\), \(C_0^\infty(\mathbb{R}^n)\) denotes the set of functions on
\(\mathbb{R}^n\) that are smooth and have compact support, and
\(\|\cdot\|_p\) denotes the usual \(L_p\) norm for functions on
\(\mathbb{R}^n\). 

The
\(L_p\)-variational capacity originates in nonlinear potential theory. Although \(\Phi_p\)
and \(C_p\) arise from different theories, their normalizations on the unit ball satisfy
\begin{equation*}\label{}
	\Phi_p(B^n)
	=
	\gamma_{n,p}C_p(B^n)
	=
	n\omega_n
\end{equation*}
for \(1<p<n\), where
\begin{equation*}\label{}
	\gamma_{n,p}
	=
	\left(\frac{p-1}{n-p}\right)^{p-1}.
\end{equation*}
	
	For a Lipschitz star body \(M\subset\mathbb{R}^n\) and \(p\geq 1\),
	the \(L_p\)-surface area \(S_p(M)\) of \(M\) is defined, in \cite{LXYZ}, by
	\begin{equation*}
		S_p(M)
		=
		nV_p(M,B^n)
		=
		\int_{\partial M}
		\left|x\cdot\nu_M(x)\right|^{1-p}\,dx.
	\end{equation*}
Note that for \(p>1\), the \(L_p\)-surface area \(S_p(M)\) is not finite for all
Lipschitz star bodies. The normalization is chosen so that
	\(
	S_p(B^n)=n\omega_n.
	\)

The following result, established in \cite{LXYZ}, gives sharp comparisons
between the \(L_p\)-surface area, the \(L_p\)-variational capacity, the
\(L_p\)-integral affine surface area, and volume.

\noindent\textbf{Theorem.} \cite[Lemma 1]{LXYZ} \textit{If \(M\) is a Lipschitz star body in \(\mathbb{R}^n\) and
	\(1\leq p<n\), then}
		\begin{equation}\label{1.1}
		S_p(M)
		\geq
		\left(\frac{p-1}{n-p}\right)^{p-1}C_p(M)
		\geq
		n\omega_n^{\frac{p}{n}}
		V(M)^{\frac{n-p}{n}},
	\end{equation}
	\textit{ and}
	\begin{equation}\label{1.2}
		S_p(M)
		\geq
		\Phi_p(M)
		\geq
		n\omega_n^{\frac{p}{n}}
		V(M)^{\frac{n-p}{n}}.
	\end{equation}
	\textit{Equality in \eqref{1.1} holds if \(M\) is a ball centered at the
	origin. Equality in the right-hand inequality of \eqref{1.2} holds
	for \(p>1\) if and only if \(M\) is an ellipsoid centered at the
	origin, and for \(p=1\) if and only if \(M\) is an ellipsoid
	containing the origin.}

	Taking \(p=1\) in inequality~\eqref{1.1} yields
	\begin{equation}\label{1.3}
		S_1(M)
		\geq
		C_1(M)
		\geq
		n\omega_n^{1/n}V(M)^{(n-1)/n}.
	\end{equation}
For every compact set \(E\subset\mathbb{R}^n\) with rectifiable
boundary, inequality~\eqref{1.3} holds, with equality when \(E\) is a closed
ball in \(\mathbb{R}^n\) (see \cite[p.~105]{Mazya}). Moreover, as shown in \cite[Section~6.1]{LXYZ}, every convex body
\(E\) satisfies
\begin{equation*}\label{1.4}
	S_1(E)=C_1(E).
\end{equation*}

Taking \(p=1\) in inequality~\eqref{1.2} yields
\begin{equation}\label{1.5}
	S_1(E)
	\geq
	\Phi_1(E)
	\geq
	n\omega_n^{1/n}V(E)^{(n-1)/n}.
\end{equation}
If \(E\subset\mathbb{R}^n\) has Lipschitz boundary, then equality in
the second inequality of~\eqref{1.5} holds if and only if \(E\) is an ellipsoid. For a convex body \(E\), the second inequality in~\eqref{1.5}
coincides with the classical Petty projection inequality; see, for
example,
\cite{GardnerGT,LutwakSelected,PettyIsoperimetric,Schneider}. Subsequently,
Zhang \cite{ZhangAffine} extended this inequality to
compact domains with piecewise \(C^1\) boundary. 

The second inequality in \eqref{1.1} is due to Maz'ya \cite{MazyaSobolevq}, and its
special case \(p=2\) and \(n=3\) was established by
P\'olya and Szeg\"{o} \cite{PolyaSzego}. In \cite{PolyaSzego}, they also proved the first
inequality in \eqref{1.1} for \(p=2\) and \(n=3\). For compact convex sets,
the second inequality in \eqref{1.2} was obtained in \cite{LYZ}, and an alternative
proof can be found in \cite{CampiGronchi}.

For a Lipschitz star body, the first inequality
in~\eqref{1.1} is obtained by choosing suitable test functions; see
\cite[Section~6]{LXYZ}. In contrast, the first inequality
in~\eqref{1.2} follows immediately from Jensen's inequality. The
second inequality in~\eqref{1.2} is proved using the \(L_p\) convexification of star
bodies, which is defined via the \(L_p\) convexification of level sets;
see \cite[Section~4.2]{LXYZ}.

The geometric inequalities in the preceding theorem provide a key
ingredient in the proof of the sharp convex Lorentz-Sobolev
inequalities established by Ludwig-Xiao-Zhang \cite{LXYZ}. Subsequent work established several extensions of the convex
Lorentz-Sobolev inequalities and introduced relative, general, and
higher-order affine capacities; see
\cite{Ober,FangXuZhouZhu,LiZhu,YeLuHuang,WangWang,HeLi,
	Xiao,XiaoZhang,XiaoRedux,HongYe,ZhouYe}.

Since
\[
S_1(K)=C_1(K)
\ \ \text{for all } K\in\mathcal{K}_o^n,
\]
it follows from \eqref{1.5} that
\[
\Phi_1(K)\leq C_1(K)
\ \ \text{for all } K\in\mathcal{K}_o^n,
\]
and
\[
\int_0^\infty \Phi\bigl(\langle f\rangle_t\bigr)\,dt
\leq
\int_0^\infty C\bigl(\langle f\rangle_t\bigr)\,dt
\ \ \text{for all } f\in C_0^\infty(\mathbb{R}^n).
\]
Here, \(f\colon U\to\mathbb{R}\) is locally Lipschitz on an open set
\(U\subset\mathbb{R}^n\), \(t>0\), and \(\langle f\rangle_t\) denotes
the \(L_p\) convexification of the level set \([f]_t=\{x\in U:|f(x)|\geq t\}\); see Section~\ref{sec:preliminaries}
for its precise definition.

Motivated by these observations, Ludwig-Xiao-Zhang \cite{LXYZ} posed the following question.

\noindent\textbf{Question.}\
\textit{Are the geometric inequality}
\begin{equation}\label{1.6}
\Phi_p(K)
\leq
\left(\frac{p-1}{n-p}\right)^{p-1}C_p(K)
\ \ \textit{for all }K\in\mathcal{K}_0^n
\end{equation}
\textit{and the analytic inequality}
\begin{equation}\label{1.7}
\int_0^\infty
\Phi_p\bigl(\langle f\rangle_t\bigr)\,dt
\leq
\left(\frac{p-1}{n-p}\right)^{p-1}
\int_0^\infty
C_p\bigl(\langle f\rangle_t\bigr)\,dt
\ \ \textit{for all }f\in C_0^\infty(\mathbb{R}^n)
\end{equation}
\textit{true for \(1<p<n\)}?

Our main result gives a negative answer to the question posed above,
showing that both proposed inequalities fail for all admissible values
of \(n\) and \(p\).

\begin{theorem}\label{thm:main}
		For every integer \(n\geq2\) and every \(1<p<n\), the geometric
		inequality~\eqref{1.6} and the analytic inequality~\eqref{1.7} both
		fail. More precisely:
		\begin{enumerate}
			\renewcommand{\labelenumi}{\textup{(\roman{enumi})}}
			
			\item there exists an origin-symmetric convex body
			\(K\in\mathcal{K}_o^n\) with \(C^\infty\) boundary and everywhere
			positive Gauss curvature such that
			\[
			\Phi_p(K)>\left(\frac{p-1}{n-p}\right)^{p-1}C_p(K);
			\]
			
			\item there exists a nonnegative function
			\(f\in C_c^\infty(\mathbb{R}^n)\) such that
			\[
			\int_0^\infty
			\Phi_p\bigl(\langle f\rangle_t\bigr)\,dt
			>
			\left(\frac{p-1}{n-p}\right)^{p-1}
			\int_0^\infty
			C_p\bigl(\langle f\rangle_t\bigr)\,dt.
			\]
		\end{enumerate}
	\end{theorem}
	
	Part (ii) relies on the geometric construction established in part (i). In fact, a smooth cutoff applied to its Minkowski functional allows the strict geometric inequality to be carried over exactly to the associated 
	\(L_p\)-convexified level sets.
	
	\textit{Organization of the paper.}
	Section~\ref{sec:preliminaries} collects the notation and preliminary results used throughout the paper. In Section~\ref{sec:translation}, high-degree even spherical-harmonic perturbations of the Euclidean unit ball are used to construct smooth, origin-symmetric convex bodies with everywhere positive Gauss curvature that do not satisfy the proposed geometric inequality, thereby proving part~\textup{(i)} of Theorem~\ref{thm:main}. Section~\ref{sec:analytic} derives the functional construction from their geometric counterparts by applying a smooth cutoff to the Minkowski functional and using the exact dilation identity for \(L_p\)-convexified level sets, thereby proving part~\textup{(ii)}. Finally, the appendix presents a concrete counterexample for \((n,p)=(3,2)\).
	
	\section{Preliminaries}
	\label{sec:preliminaries}
	
	In this section, we collect some basic facts and notation. For further background, we refer the reader to the monographs by Schneider \cite{Schneider} and Gardner \cite{GardnerGT} and the paper \cite{LXYZ}.
	
	Let \(d\sigma\) denote the standard surface measure on
	\(\mathbb{S}^{n-1}\), and write
	\[
	M:=\sigma(\mathbb{S}^{n-1})
	=\lvert\mathbb{S}^{n-1}\rvert
	=n\omega_n.
	\]
	For \(q\in L^1(\mathbb{S}^{n-1})\), define its normalized spherical
	average by
	\[
	\langle q\rangle_\sigma
	:=
	\frac{1}{M}
	\int_{\mathbb{S}^{n-1}}q\,d\sigma.
	\]
	
Let \(K\in\mathcal{K}_o^n\). Its Minkowski functional is
	defined by
	\[
	r_K(x):=\inf\{s>0:x\in sK\},
	\ \ x\in\mathbb{R}^n.
	\]
	For every \(u\in\mathbb{S}^{n-1}\), we have
	\[
	r_K\bigl(\rho_K(u)u\bigr)=1.
	\]
	Moreover,
	\[
	r_K(x)=\frac{1}{\rho_K(x)},
	\ \
	x\in\mathbb{R}^n\setminus\{0\}.
	\]

For \(K,L\in\mathcal{K}_o^n\), Lutwak
\cite{LutwakFirey} established the representation
\begin{equation*}\label{2.1}
	V_p(K,L)
	=
	\frac{1}{n}
	\int_{\mathbb{S}^{n-1}}
	h_L(u)^p\,dS_p(K,u),
\end{equation*}
where \(S_p(K,\cdot)\) denotes the \(L_p\)-surface area measure of
\(K\), defined by
\[
dS_p(K,u)
=
h_K(u)^{1-p}\,dS_K(u).
\]
Here \(S_K\) is the classical surface area measure of \(K\); see
\cite{Schneider}. Equivalently, \(S_K\) is the Borel measure on
\(\mathbb{S}^{n-1}\) characterized by
\begin{equation*}\label{2.2}
	\int_{\mathbb{S}^{n-1}}g(u)\,dS_K(u)
	=
	\int_{\partial K}
	g\bigl(\nu_K(x)\bigr)\,
	d\mathcal{H}^{n-1}(x)
\end{equation*}
for every \(g\in C(\mathbb{S}^{n-1})\).
	
Let \(K\in\mathcal{K}_o^n\). For each
\(u\in\Sph^{n-1}\), denote by
\(
\bar{u}=[-u/2,u/2]
\)
the centered line segment in the direction \(u\). Since
\[
h_{\bar{u}}(v)=\frac{|u\cdot v|}{2}, \ \ v\in \Sph^{n-1},
\]
the \(L_p\)-mixed volume of \(K\) and \(\bar{u}\) is given by
\begin{equation}\label{eq:Vp}
	V_p(K,\bar{u})
	=
	\frac{1}{n}
	\int_{\Sph^{n-1}}
	\left(\frac{|u\cdot v|}{2}\right)^p
	\,dS_p(K,v).
\end{equation}
Let
\[
\beta_{n,p}=V_p(B^n,\bar{u}).
\]
By rotational invariance, \(\beta_{n,p}\) is independent of the choice
of \(u\). With the normalization adopted in \cite{LXYZ}, we have
\begin{align}
	S_p(K)
	&=
	\frac{1}{\beta_{n,p}}
	\int_{\Sph^{n-1}}
	V_p(K,\bar{u})\,d\sigma(u),
	\label{eq:Sp-average}\\
	\Phi_p(K)
	&=
	\frac{M^{(n+p)/n}}{\beta_{n,p}}
	\left(
	\int_{\Sph^{n-1}}
	V_p(K,\bar{u})^{-n/p}\,d\sigma(u)
	\right)^{-p/n}.
	\label{eq:Phi}
\end{align}
	In particular,
\begin{equation*}\label{eq:ball-Phi}
	\Phi_p(B^n)=S_p(B^n)=M.
\end{equation*}

Define the normalized \(p\)-cosine transform of an
integrable function \(q\) on \(\Sph^{n-1}\) by
\begin{equation*}\label{eq:cosine-transform}
	(\mathcal{T}_pq)(u)
	:=
	\frac{
		\displaystyle
		\int_{\Sph^{n-1}}
		|u\cdot v|^p q(v)\,d\sigma(v)
	}{
		\displaystyle
		\int_{\Sph^{n-1}}
		|e_1\cdot v|^p\,d\sigma(v)
	}.
\end{equation*}
This normalization ensures that
\[
\mathcal{T}_p1=1,
\qquad
\avg{\mathcal{T}_pq}=\avg{q}.
\]
For a smooth strictly convex body \(K\), let
\[
dS_p(K,u)=q_K(u)\,d\sigma(u)
\]
and
\[
v_K(u)=\frac{V_p(K,\bar{u})}{\beta_{n,p}}.
\]
Then
\begin{equation}\label{eq:normalized-functionals}
	v_K=\mathcal{T}_pq_K,
	\qquad
	\frac{S_p(K)}{M}=\avg{v_K},
	\ \
	\frac{\Phi_p(K)}{M}
	=
	\left\langle
	v_K^{-n/p}
	\right\rangle_{\!\sigma}^{-p/n}.
\end{equation}	
	
For \(1\leq p<n\) and a compact set \(E\subset\mathbb{R}^n\), define
its \(L_p\)-variational capacity by
\begin{equation}\label{2.9}
C_p(E)
=
\inf\left\{
\int_{\mathbb{R}^n}|\nabla\varphi|^p\,dx:
\varphi\in C_c^\infty(\mathbb{R}^n),\
\varphi\geq1\ \text{on }E
\right\}.
\end{equation}	
For the unit
ball \(B^n\),
\[
C_p(B^n)
=
\begin{cases}
	M, & p=1,\\[4pt]
	M
	\left(\dfrac{n-p}{p-1}\right)^{p-1},
	& 1<p<n.
\end{cases}
\]
The
\(L_p\)-variational capacity is monotone with respect to set
inclusion and invariant under translations. 
	Both $\Phi_p$ and $\pcap_p$ are homogeneous of degree $n-p$, and hence
	\begin{equation*}\label{eq:normalization}
		\gammaconst\,\pcap_p(B^n)=\Phi_p(B^n)=M.
	\end{equation*}

	Let \(f\in C_0^\infty(\R^n)\), and define its superlevel sets by
	\[
	[f]_t
	=
	\left\{
	x\in\R^n:
	\lvert f(x)\rvert\geq t
	\right\}.
	\]
	For almost every \(t>0\), both \(t\) and \(-t\) are regular values of
	\(f\). For such \(t\) with \([f]_t\neq\varnothing\), the outer
	unit normal to \([f]_t\) is
	\[
	\nu_f(x)
	=
	-\frac{\nabla\lvert f\rvert(x)}
	{\lvert\nabla f(x)\rvert}
	=
	-\operatorname{sgn}\bigl(f(x)\bigr)
	\frac{\nabla f(x)}
	{\lvert\nabla f(x)\rvert},
	\ \
	x\in\partial[f]_t.
	\]
	The \(L_p\) convexification \(\convlev{f}{t}\) is defined as the unique
	origin-symmetric convex body whose \(L_p\)-surface area measure
	satisfies
	\begin{equation}\label{eq:convexification-definition}
		\int_{\Sph^{n-1}}
		\phi(u)\,
		dS_p\bigl(\convlev{f}{t},u\bigr)
		=
		\int_{\partial[f]_t}
		\phi\bigl(\nu_f(x)\bigr)
		\lvert\nabla f(x)\rvert^{p-1}
		\,d\Haus^{n-1}(x)
	\end{equation}
	for every even function
	\(\phi\in C(\Sph^{n-1})\); see
	\cite[Section~4.1]{LXYZ}.

	\section{Proof of (\textnormal{i}) in Theorem \ref{thm:main}}
	\label{sec:translation}
		
	Fix an integer \(n\geq2\) and a real number \(p\in(1,n)\), and take
	\begin{equation*}\label{3.1}
		k=n-p,
		\ \
		\mu=\frac{k}{p-1}.
	\end{equation*}
	Choose a nonzero real spherical harmonic \(Y_\ell\) of even degree
	\(\ell\geq2\). It satisfies
	\begin{equation*}\label{3.2}
		-\Delta_{\mathbb{S}^{n-1}}Y_\ell
		=
		\lambda_\ell Y_\ell \ \ \text{where} \ \ 
		\lambda_\ell:=\ell(\ell+n-2).
	\end{equation*}
	For an even integer \(\ell\geq 2\) and sufficiently small
	\(\lvert\varepsilon\rvert\), let
	\begin{equation*}\label{3.3}
		h_\varepsilon=1+\varepsilon Y_\ell.
	\end{equation*}
	Then \(h_\varepsilon>0\) and
	\[
	\nabla_{\mathbb{S}^{n-1}}^2h_\varepsilon+h_\varepsilon I
	>0
	\]
	on \(\mathbb{S}^{n-1}\). Thus, \(h_\varepsilon\) is the support
	function of a \(C^\infty\) convex body \(K_\varepsilon\) with everywhere
	positive Gauss curvature. Moreover, since \(Y_\ell\) is even, we have
	\(
	h_\varepsilon(-u)=h_\varepsilon(u).
	\)
	Hence, \(K_\varepsilon\) is origin-symmetric.
	
	The following lemma gives the second-order expansion of
	\(S_p(K_\varepsilon)-\Phi_p(K_\varepsilon)\).

	\begin{lemma}\label{lem:affine-deficit}
	Let \(\ell=2j\geq2\) be a fixed even integer. Define
		\begin{equation}\label{eq:cosine-eigenvalue}
			a_{p,2j}
			=
			\prod_{r=0}^{j-1}
			\frac{p-2r}{n+p+2r},
		\ \
			A_\ell
			=
			(k-\lambda_\ell)a_{p,\ell}.
		\end{equation}
	Then as \(\varepsilon\to0\),
		\begin{equation}\label{eq:affine-deficit}
			S_p(K_\varepsilon)-\Phi_p(K_\varepsilon)
			=
			\frac{n+p}{2p}A_\ell^2\varepsilon^2
			\int_{\mathbb{S}^{n-1}}Y_\ell^2\,d\sigma
			+o(\varepsilon^2).
		\end{equation}
	\end{lemma}
	
	\begin{proof}
		
			Take
			\[
			H
			=
			\nabla_{\mathbb{S}^{n-1}}^2Y_\ell+Y_\ell I,
			\ \
			\tau
			=
			\operatorname{tr}H
			=
			(n-1-\lambda_\ell)Y_\ell.
			\]
			Since \(Y_\ell\) and \(H\) are bounded on
			\(\mathbb{S}^{n-1}\), there exists \(\varepsilon_0>0\) such that whenever \(\lvert\varepsilon\rvert<\varepsilon_0\),
			\[
			h_\varepsilon\geq\frac12,
			\ \
			I+\varepsilon H\geq \frac12 I
			\quad\text{on }\mathbb{S}^{n-1}.
			\]
	It follows that all expansions below are uniform for
		\(u\in\mathbb{S}^{n-1}\). Using the Gauss parametrization, the
		\(L_p\) surface area measure of \(K_\varepsilon\) can be written as
		\[
		dS_p(K_\varepsilon,u)
		=
		q_\varepsilon(u)\,d\sigma(u),
		\ \ 
		q_\varepsilon(u)
		:=
		h_\varepsilon(u)^{1-p}
		\det\bigl(I+\varepsilon H(u)\bigr).
		\]
Using the expansions
		\begin{align*}
			h_\varepsilon^{1-p}
			&=1+(1-p)\varepsilon Y_\ell
			+\frac{p(p-1)}2\varepsilon^2Y_\ell^2
			+O_{C^0}(\varepsilon^3),
		\end{align*}
		and
	 \begin{align*}
			\det(I+\varepsilon H)
			&=1+\varepsilon\tau
			+\frac{\varepsilon^2}{2}
			\bigl(\tau^2-\operatorname{tr}(H^2)\bigr)
			+O_{C^0}(\varepsilon^3),
		\end{align*}
		we obtain
		\begin{equation}\label{eq:surface-density-first}
			q_\varepsilon
			=1+\varepsilon(k-\lambda_\ell)Y_\ell
			+\varepsilon^2R_0+O_{C^0}(\varepsilon^3),
		\end{equation}
		where
		\[
		R_0=
		\frac12\bigl(\tau^2-\operatorname{tr}(H^2)\bigr)
		+(1-p)Y_\ell\tau+\frac{p(p-1)}2Y_\ell^2.
		\]
		
		The Funk-Hecke formula \cite[Chapter~3]{Groemer}
		(or a direct Fourier computation when \(n=2\)) shows that
		\(Y_{2j}\) is an eigenfunction of \(\mathcal{T}_p\) with eigenvalue
		\(a_{p,2j}\). That is,
		\begin{equation*}\label{eq:Funk-Hecke}
			\mathcal{T}_pY_{2j}
			=
			a_{p,2j}Y_{2j}.
		\end{equation*}
		For \(j\geq0\), a direct computation of the corresponding beta integrals gives
		\[
		a_{p,0}=1,
	\ \
		a_{p,2j+2}
		=
		\frac{p-2j}{n+p+2j}\,a_{p,2j}.
		\]
		Repeated application of this recurrence yields the product formula
		in \eqref{eq:cosine-eigenvalue}. In particular,
		\[
		a_{p,2}=\frac{p}{n+p}.
		\]
		Applying \(\mathcal{T}_p\) to the expansion in
		\eqref{eq:surface-density-first}, and using the boundedness of
		\(
		\mathcal{T}_p:
		C^0(\Sph^{n-1})\longrightarrow C^0(\Sph^{n-1}),
		\)
		we obtain
		\[
		v_{K_\varepsilon}
		=
		1+\varepsilon A_\ell Y_\ell
		+\varepsilon^2B_0
		+O_{C^0}(\varepsilon^3),
		\ \
		B_0:=\mathcal{T}_pR_0.
		\]
		Since \(\mathcal{T}_p\) preserves normalized spherical averages, we have
		\[
		\avg{B_0}=\avg{R_0}.
		\]
		Moreover, \(\avg{Y_\ell}=0\). Therefore, the formula for \(S_p\) in
		\eqref{eq:normalized-functionals} gives
		\[
		\frac{S_p(K_\varepsilon)}{M}
		=
		1+\varepsilon^2\avg{B_0}
		+O(\varepsilon^3).
		\]
		Let $\alpha=n/p$.
		Taylor expanding \(v_{K_\varepsilon}^{-\alpha}\), and applying the
		power \(-1/\alpha\) to its spherical average as in
		\eqref{eq:normalized-functionals}, we obtain
		\begin{align*}
			\avg{v_{K_\varepsilon}^{-\alpha}}
			&=
			1-\alpha\varepsilon^2\avg{B_0}
			+\frac{\alpha(\alpha+1)}{2}
			A_\ell^2\varepsilon^2\avg{Y_\ell^2}
			+O(\varepsilon^3),
		\end{align*}
		and
		\begin{align*}
			\frac{\Phi_p(K_\varepsilon)}{M}
			&=
			1+\varepsilon^2
			\left(
			\avg{B_0}
			-\frac{\alpha+1}{2}
			A_\ell^2\avg{Y_\ell^2}
			\right)
			+O(\varepsilon^3).
		\end{align*}
		Combining the second expansion with the preceding expansion of
		\(S_p(K_\varepsilon)/M\), and using
		\[
		\alpha+1=\frac{n+p}{p} \ \
		\text{and} \ \
		M\avg{Y_\ell^2}
		=
		\int_{\Sph^{n-1}}Y_\ell^2\,d\sigma,
		\]
		we give \eqref{eq:affine-deficit}, in fact with an $O(\varepsilon^3)$ remainder
		for fixed $\ell$.
	\end{proof}
	
	Let $\rho_\varepsilon$ be the radial function of $K_\varepsilon$.
	The support parametrization
	\[
	x(u)=h_\varepsilon(u)u+\nabla_{\Sph}h_\varepsilon(u)
	\]
	gives
	\begin{equation*}\label{eq:rho-first}
		\rho_\varepsilon
		=
		1+\varepsilon Y_\ell+O_{C^2}(\varepsilon^2) \ \
	\text{and} \ \
		b_\varepsilon=\nabla_{\Sph}\log\rho_\varepsilon
		=
		\varepsilon\nabla_{\Sph}Y_\ell+O_{C^1}(\varepsilon^2).
	\end{equation*}
	
We next establish the corresponding second-order lower bound for the
capacity deficit
\(
S_p(K_\varepsilon)-\gamma_{n,p}C_p(K_\varepsilon).
\)	
	\begin{lemma}\label{lem:capacity-deficit}
		Let $s_\ell>\mu$ be defined by
		\begin{equation}\label{eq:s-ell}
			(p-1)s_\ell(s_\ell-\mu)=\lambda_\ell,
			\ \
			s_\ell=
			\frac{\mu+\sqrt{\mu^2+4\lambda_\ell/(p-1)}}2.
		\end{equation}
		Then
		\begin{equation}\label{eq:capacity-deficit}
			S_p(K_\varepsilon)-\gammaconst\,\pcap_p(K_\varepsilon)
			\geq
			\frac{p(p-1)}2(s_\ell-\mu)^2\varepsilon^2
			\int_{\Sph^{n-1}}Y_\ell^2\,d\sigma
			+o(\varepsilon^2).
		\end{equation}
	\end{lemma}
	
	\begin{proof}
		 For $\theta\in\Sph^{n-1}$ and $t\geq1,$ we parametrize the exterior of \(K_\varepsilon\) by
		\[
		x=\rho_\varepsilon(\theta)t\theta.
		\]
		Under this parametrization, write
		\[
		v\bigl(\rho_\varepsilon(\theta)t\theta\bigr)
		=
		F(t,\theta).
		\]
		The corresponding Jacobian is
		\[
		dx
		=
		\rho_\varepsilon(\theta)^n t^{n-1}
		\,dt\,d\sigma(\theta),
		\]
		while the chain rule gives
		\[
		\ \ \nabla v
		=
		\rho_\varepsilon(\theta)^{-1}
		\left[
		F_t e_r
		+\frac{1}{t}
		\left(
		\nabla_{\Sph}F
		-tF_tb_\varepsilon
		\right)
		\right],
		\ \ e_r=\theta.
		\]
		Thus a direct polar-coordinate calculation gives
		\begin{align}
			E_\varepsilon(F)
			&:=
			\int_{\R^n\setminus K_\varepsilon}|\nabla v|^p\,dx \nonumber\\
			&=
			\int_{\Sph^{n-1}}\int_1^\infty
			\rho_\varepsilon^k t^{n-1}
			\left[
			F_t^2+t^{-2}
			\left|\nabla_{\Sph}F-tF_tb_\varepsilon\right|^2
			\right]^{p/2}
			dt\,d\sigma.
			\label{eq:exact-energy}
		\end{align}
		Consider the function
		\[
		F_T(t,\theta):=t^{-\mu},
		\ \ t\geq1,
		\]
		extended by the constant \(1\) on \(K_\varepsilon\). It
		satisfies
		\begin{equation}\label{eq:gauge-energy}
			E_\varepsilon(F_T)
			=
			\mu^{p-1}S_p(K_\varepsilon).
		\end{equation}
		To verify this, observe that on \(\partial K_\varepsilon\),
		\[
		\left|
		\nabla r_{K_\varepsilon}(x)
		\right|
		=
		h_{K_\varepsilon}
		\bigl(\nu_{K_\varepsilon}(x)\bigr)^{-1}.
		\]
		Therefore, the radial representation of
		\(\partial K_\varepsilon\) yields
		\[
		S_p(K_\varepsilon)
		=
		\int_{\Sph^{n-1}}
		\rho_\varepsilon^k
		\left(1+|b_\varepsilon|^2\right)^{p/2}
		\,d\sigma.
		\]
		Substituting \(F_T\) into \eqref{eq:exact-energy} and evaluating the
		radial integral,
		\[
		\mu^p
		\int_1^\infty
		t^{n-1-p(\mu+1)}\,dt
		=
		\mu^{p-1},
		\]
		we obtain \eqref{eq:gauge-energy}.
		
	We refine the preceding function by defining
	\begin{equation*}\label{eq:improved-competitor}
		F_*(t,\theta)
		=
		t^{-\mu}
		+\varepsilon\mu Y_\ell(\theta)
		\bigl(t^{-s_\ell}-t^{-\mu}\bigr).
	\end{equation*}
	Since \(F_*(1,\theta)=1\), we extend \(F_*\) to \(K_\varepsilon\)
	by defining
	\(
	F_*\equiv1
	\text{ on }K_\varepsilon.
	\)
The extended function has trace \(1\) on
\(\partial K_\varepsilon\) and is an admissible finite-energy test
function in the corresponding homogeneous Sobolev space. To approximate \(F_*\) by admissible functions in the capacity
definition \eqref{2.9}, fix \(\delta>0\) and modify \(F_*\) near
the boundary by setting it equal to \(1\) for
\(1\leq t\leq1+\delta\). In the transition region
\(1+\delta<t<1+2\delta\), interpolate with respect to \(t\) so that
the modified function equals \(1\) at \(t=1+\delta\) and
\(F_*(1+2\delta,\theta)\) at \(t=1+2\delta\). 
Since
\[
F_*(t,\theta)-1=O(t-1),
\ \ \text{as }t\downarrow1,
\]
uniformly in \(\theta\), the interpolation can be chosen with bounded
gradient. Its contribution to the energy is therefore \(O(\delta)\).

Next, choose a smooth radial cutoff that equals \(1\) on
\(\{|x|\leq R\}\) and vanishes on \(\{|x|\geq2R\}\). The energy of the
discarded tail and the additional term arising from the gradient of
the cutoff are both bounded by
\[
O\bigl(R^{\,n-p(\mu+1)}\bigr)
=
O(R^{-\mu})
=
o(1),
\ \ \text{as }R\to\infty.
\]
The truncated Sobolev function is compactly supported and remains
identically equal to \(1\) on a neighborhood of \(K_\varepsilon\).
Mollifying it at a scale smaller than the width of this neighborhood
produces a function in \(C_c^\infty(\R^n)\) that is still equal to
\(1\) on \(K_\varepsilon\), while its energy converges to that of the
truncated function. Let the mollification scale tend to zero. Then \(R\to\infty\), and
finally \(\delta\to0\). Thus we obtain
\begin{equation}\label{eq:capacity-upper-competitor}
	\pcap_p(K_\varepsilon)
	\leq
	E_\varepsilon(F_*).
\end{equation}
This approximation argument is the standard passage between the
homogeneous Sobolev and smooth definitions of capacity; see
\cite[Chapter~2]{Mazya}.
		
We next compute the second-order change in energy. For \(a>0\), let
\[
u(t)=t^{-\mu},
\ \
w_a(t,\theta)=\mu t^{-a}Y_\ell(\theta),
\ \
T_a=t^{-1}\nabla_{\Sph}w_a.
\]
Define
\[
Z=w_{s_\ell}-w_\mu.
\]
Then the improved function can be written as
\[
F_*=u+\varepsilon Z.
\]		
Let \(c\neq0\), and decompose
\[
z=z_re_r+z_{\mathrm{tan}}
\ \ \text{and} \ \
z_{\mathrm{tan}}\perp e_r,
\]
with \(|z|\leq C|c|\). After factoring out \(|c|\), the normalized
argument remains in a fixed compact set away from the origin for all
sufficiently small \(|\varepsilon|\). Taylor's theorem then yields the following uniform expansion:
\begin{align}\label{eq:vector-taylor}
	|ce_r+\varepsilon z|^p
	={}&
	|c|^p
	+p|c|^{p-2}cz_r\varepsilon
	\nonumber\\
	&+
	\frac{p}{2}|c|^{p-2}
	\left(
	(p-1)z_r^2+|z_{\mathrm{tan}}|^2
	\right)\varepsilon^2
	+O\bigl(\varepsilon^3|c|^p\bigr).
\end{align}		
To isolate the second-order angular terms, define
\[
\widetilde b_\varepsilon
=
b_\varepsilon
-\varepsilon\nabla_{\Sph}Y_\ell
=
O_{C^1}(\varepsilon^2).
\]
The gradient vectors appearing in the integrand of
\eqref{eq:exact-energy} can then be written as
\begin{align*}
	\mathcal{A}_T
	&=
	\left(
	u',
	\varepsilon T_\mu-u'\widetilde b_\varepsilon
	\right),\\
	\mathcal{A}_*
	&=
	\left(
	u'+\varepsilon Z_t,
	\varepsilon T_{s_\ell}
	-\varepsilon^2Z_t\nabla_{\Sph}Y_\ell
	-u'\widetilde b_\varepsilon
	+O(\varepsilon^3|u'|)
	\right).
\end{align*}
Thus, relative to the radial base vector \((u',0)\), their respective
first-order perturbations are
\[
(0,T_\mu)
\ \ \text{and} \ \
(Z_t,T_{s_\ell}).
\]
Moreover, the radial factor satisfies
\[
\rho_\varepsilon^k
=
1+k\varepsilon Y_\ell
+O(\varepsilon^2).
\]
The order-\(\varepsilon\) term in
\(E_\varepsilon(F_T)-E_\varepsilon(F_*)\) vanishes, since for each
fixed \(t\), \(Z_t(t,\cdot)\) is a scalar multiple of \(Y_\ell\),
whose spherical mean is zero.

Let
\[
W(t)=t^{n-1}\lvert u'(t)\rvert^{p-2}.
\]
Applying the vector expansion \eqref{eq:vector-taylor} to the exact
energy formula \eqref{eq:exact-energy}, we obtain the following
unreduced second-order expansion:
\begin{align}
	E_\varepsilon(F_T)-E_\varepsilon(F_*)
	={}&
	\frac{p}{2}\varepsilon^2
	\int_{\Sph^{n-1}}\int_1^\infty
	W(t)
	\left(
	\lvert T_\mu\rvert^2
	-\lvert T_{s_\ell}\rvert^2
	-(p-1)Z_t^2
	\right)
	\,dt\,d\sigma
	\nonumber\\
	&\quad
	-pk\varepsilon^2
	\int_{\Sph^{n-1}}\int_1^\infty
	W(t)u'Y_\ell Z_t
	\,dt\,d\sigma
	+o(\varepsilon^2).
	\label{eq:raw-energy-difference}
\end{align}
The second integral arises precisely from the interaction between the
first-order term \(k\varepsilon Y_\ell\) in the expansion of
\(\rho_\varepsilon^k\) and the first-order radial variation of the
energy density \(\lvert\nabla F_*\rvert^p\). The terms involving the
second-order contribution to \(\rho_\varepsilon^k\) are identical for
the two energies and therefore cancel. Moreover, the second-order
correction vectors
\[
-u'\widetilde b_\varepsilon
\ \ \text{and}\ \
-\varepsilon^2Z_t\nabla_{\Sph}Y_\ell
\]
are both tangential and hence orthogonal to the radial base gradient.
Their linear pairings with the base gradient therefore vanish. Hence, they
produce no additional contribution at order \(\varepsilon^2\).

		Define 
		\begin{align}
			Q(w)
			=
			\int_{\Sph^{n-1}}\int_1^\infty
			&\mu^{p-2}
			t^{n-1-(\mu+1)(p-2)}
			\nonumber\\[-2mm]
			&\times
			\left[(p-1)w_t^2+t^{-2}|\nabla_{\Sph}w|^2\right]
			dt\,d\sigma.
			\label{eq:Q}
		\end{align}
		Since the weight in \eqref{eq:Q} is $W$, the difference between the
		coefficient in \eqref{eq:raw-energy-difference} and
		$\frac p2[Q(w_\mu)-Q(w_{s_\ell})]$ equals
		\[
		p(p-1)\int_{\Sph^{n-1}}\int_1^\infty
		W(w_\mu)_tZ_t\,dt\,d\sigma
		-pk\int_{\Sph^{n-1}}\int_1^\infty
		Wu'Y_\ell Z_t\,dt\,d\sigma.
		\]
		This vanishes pointwise in the remaining factors since
		\[
		(w_\mu)_t=\mu Y_\ell u',
		\ \ k=(p-1)\mu.
		\]
		Therefore, we have proved
		\begin{equation}\label{eq:energy-difference-Q}
			E_\varepsilon(F_T)-E_\varepsilon(F_*)
			=
			\frac p2\varepsilon^2
			\bigl[Q(w_\mu)-Q(w_{s_\ell})\bigr]
			+o(\varepsilon^2).
		\end{equation}
		
		The same calculation, with the elementary $t$-integrals performed, displays
		the complete second-order coefficient as
		\begin{align*}
			E_\varepsilon(F_T)-E_\varepsilon(F_*)
			={}&
			\frac p2\mu^p\varepsilon^2
			\int_{\Sph^{n-1}}Y_\ell^2\,d\sigma
			\nonumber\\
			&\times
			\left\{
			\lambda_\ell
			\left(\frac1\mu-\frac1{2s_\ell-\mu}\right)
			-(p-1)
			\left(\frac{s_\ell^2}{2s_\ell-\mu}-\mu\right)
			\right\}
			+o(\varepsilon^2).
			\label{eq:raw-energy-coefficient}
		\end{align*}
		Since \(t^{\mu-s_\ell}\leq1\) for \(t\geq1\), the ratios of all
		perturbation fields to \(\lvert u'\rvert\) are uniformly bounded.
		Thus, for fixed \(\ell\) and sufficiently small
		\(\lvert\varepsilon\rvert\), every line segment joining the base
		gradient to the corresponding perturbed gradient remains at least
		\(\lvert u'\rvert/2\) away from the origin.
		
		The third derivative of $z\mapsto|z|^p$ on this region
		is bounded by $C|u'|^{p-3}$. Hence, the pointwise remainder is at most
		$C|\varepsilon|^3|u'|^p$.  Thus, we have
		\[
		\int_1^\infty t^{n-1}|u'|^p\,dt<\infty.
		\]
		The dominated convergence theorem shows that the remainder in
		\eqref{eq:raw-energy-difference} is \(o(\varepsilon^2)\). The same
		argument remains valid when \(1<p<2\).
		
		Using
		\[
		\int_{\Sph^{n-1}}
		\lvert\nabla_{\Sph}Y_\ell\rvert^2\,d\sigma
		=
		\lambda_\ell
		\int_{\Sph^{n-1}}Y_\ell^2\,d\sigma,
		\]
		and evaluating the radial integral in \eqref{eq:Q}, we obtain
		\begin{equation*}\label{eq:Q-mode}
			Q(w_\tau)
			=
			\mu^p
			\frac{(p-1)\tau^2+\lambda_\ell}{2\tau-\mu}
			\int_{\Sph^{n-1}}Y_\ell^2\,d\sigma.
		\end{equation*}
		Using \eqref{eq:s-ell}, it follows that
		\begin{equation}\label{eq:Q-gap}
			Q(w_\mu)-Q(w_{s_\ell})
			=
			\mu^{p-1}(p-1)(s_\ell-\mu)^2
			\int_{\Sph^{n-1}}Y_\ell^2\,d\sigma.
		\end{equation}
		Combining \eqref{eq:capacity-upper-competitor} with
		\(\gammaconst=\mu^{1-p}\), we obtain
		\[
		S_p(K_\varepsilon)
		-\gammaconst\pcap_p(K_\varepsilon)
		\geq
		\mu^{1-p}
		\left[
		E_\varepsilon(F_T)-E_\varepsilon(F_*)
		\right].
		\]
		The desired estimate \eqref{eq:capacity-deficit} now follows from
		\eqref{eq:energy-difference-Q} and \eqref{eq:Q-gap}.
	\end{proof}
	
	The preceding estimates yield the following counterexample.
	
	\begin{proposition}\label{prop:symmetric-counterexample}
		For every $n\geq2$ and $1<p<n$, there is a sufficiently high even degree
		$\ell$ such that for sufficiently small \(\varepsilon\neq0\),
		\begin{equation*}\label{eq:symmetric-reverse}
			\Phi_p(K_\varepsilon)>
			\gammaconst\,\pcap_p(K_\varepsilon).
		\end{equation*}
	\end{proposition}
	
	\begin{proof}
		Combining Lemmas~\ref{lem:affine-deficit} and
		\ref{lem:capacity-deficit}, we obtain
		\begin{align}
			\Phi_p(K_\varepsilon)
			-\gammaconst\pcap_p(K_\varepsilon)
			\geq{}&
			\frac{\varepsilon^2}{2}
			\left[
			p(p-1)(s_\ell-\mu)^2
			-\frac{n+p}{p}A_\ell^2
			\right]
			\nonumber\\
			&\times
			\int_{\Sph^{n-1}}Y_\ell^2\,d\sigma
			+o(\varepsilon^2).
			\label{eq:master-gap}
		\end{align}
		It therefore suffices to choose \(\ell\) so that the coefficient in
		brackets is positive.
		
		If \(p=2q\) is an even integer, then
		\eqref{eq:cosine-eigenvalue} immediately yields
		\[
		a_{p,2j}=0
		\ \ \text{for every }j>q.
		\]
		If \(p\) is not an even integer, the product formula can instead be
		expressed in terms of Gamma functions as
		\begin{equation*}\label{eq:gamma-eigenvalue}
			a_{p,2j}
			=
			(-1)^j
			\frac{
				\Gamma(j-p/2)\Gamma((n+p)/2)
			}{
				\Gamma(-p/2)\Gamma(j+(n+p)/2)
			}.
		\end{equation*}
		The standard asymptotic formula for ratios of Gamma functions yields
		\[
		\lvert a_{p,2j}\rvert
		=
		O\bigl(j^{-p-n/2}\bigr),
		\ \
		A_{2j}
		=
		O\bigl(j^{2-p-n/2}\bigr)
		=
		o(1),
		\]
		where the last relation follows from \(n+2p>4\). By contrast, the definition of \(s_{2j}\) gives
		\[
		s_{2j}-\mu
		=
		\frac{
			\sqrt{\mu^2+4\lambda_{2j}/(p-1)}-\mu
		}{2}
		\sim
		\frac{2j}{\sqrt{p-1}}
		\ \ \text{as }j\to\infty.
		\]
		Therefore, when \(p\) is not an even integer, the expression in
		brackets in \eqref{eq:master-gap} satisfies
		\[
		p(p-1)(s_{2j}-\mu)^2
		-\frac{n+p}{p}A_{2j}^2
		=
		4pj^2\bigl(1+o(1)\bigr).
		\]
		When \(p=2q\) is an even integer, \(A_{2j}=0\) for \(j>q\), while the
		first term in the bracket is strictly positive. Hence, in either case,
		the bracket is positive for every sufficiently large \(j\). Fix such
		a \(j\), and let \(\ell=2j\). Since the remainder in
		\eqref{eq:master-gap} is \(o(\varepsilon^2)\) for this fixed degree,
		the desired strict inequality holds for every sufficiently small
		\(\varepsilon\neq0\). 
	\end{proof}
		\begin{proof}[Proof of (i) in Theorem \ref{thm:main}]
	By Proposition~\ref{prop:symmetric-counterexample}, we obtain
	part~\textup{(i)} of Theorem~\ref{thm:main}. In particular, the
	geometric inequality fails even for origin-symmetric convex bodies
	with \(C^\infty\) boundary and everywhere positive Gauss curvature.
		\end{proof}
	\section{Proof of (\textnormal{ii}) in Theorem \ref{thm:main}}
	\label{sec:analytic}
	
	We introduce the standard flat functionin
	\[
	\vartheta(s)
	=
	\begin{cases}
		\exp(-1/s), & s>0,\\
		0, & s\leq0,
	\end{cases}
	\]
	and define
	\begin{equation*}\label{eq:flat-cutoff}
		g(s)
		=
		\frac{\vartheta(2-s)}
		{\vartheta(2-s)+\vartheta(s-1)}.
	\end{equation*}
	Then \(g\in C^\infty(\R)\), with
	\[
	g\equiv1 \ \ \text{on }(-\infty,1],
	\ \ 
	g\equiv0 \ \ \text{on }[2,\infty),
	\]
	and \(g'(s)<0\) for every \(s\in(1,2)\).

The following proposition transfers the preceding geometric
construction to the corresponding functional setting.
	
	\begin{proposition}\label{prop:analytic-counterexample}
		Let \(K\in\Kzero\) be a \(C^\infty\), origin-symmetric convex body
		satisfying
		\[
		\Phi_p(K)
		>
		\gammaconst\,\pcap_p(K).
		\]
		Define
		\begin{equation}\label{eq:analytic-test}
			f(x):=g\bigl(r_K(x)\bigr).
		\end{equation}
		Then \(f\) is a nonnegative function in \(C_0^\infty(\R^n)\) and
		satisfies
		\begin{equation}\label{eq:analytic-reverse}
			\int_0^\infty
			\Phi_p\bigl(\convlev{f}{t}\bigr)\,dt
			>
			\gammaconst
			\int_0^\infty
			\pcap_p\bigl(\convlev{f}{t}\bigr)\,dt.
		\end{equation}
	\end{proposition}
	
	\begin{proof}
		The function \(r_K\) is smooth on
		\(\R^n\setminus\{0\}\). Since \(g\equiv1\) on
		\((-\infty,1]\) and vanishes to infinite order at \(s=2\), the function \(f=g\circ r_K\) is smooth at the origin and matches
		smoothly with the zero function along the boundary of its support. Hence, the function $f$ defined by \eqref{eq:analytic-test} belongs to $C_0^\infty(\R^n)$.
		
		Fix \(t\in(0,1)\). Since \(g\) is strictly decreasing on \((1,2)\),
		there exists a unique \(s=s(t)\in(1,2)\) such that \(t=g(s)\).
		Applying the computation from the proof of
		\cite[Lemma~8]{LXYZ} to this regular level yields
		\begin{equation}\label{eq:level-dilate}
			\convlev{f}{t}
			=
			c(t)K,
			\ \
			c(t)^{n-p}
			=
			\lvert g'(s)\rvert^{p-1}s^{n-1}.
		\end{equation}
		Indeed, the level surface \(\{f=t\}\) is \(s\partial K\), and on this
		surface,
		\[
		\lvert\nabla f\rvert
		=
		\lvert g'(s)\rvert\,\lvert\nabla r_K\rvert.
		\]
		Moreover, the \((n-1)\)-dimensional surface measure acquires the factor
		\(s^{n-1}\) under dilation. Thus, the Gauss map \(\nu_K\) satisfies
		\[
		(\nu_K)_*\!\left(
		\lvert\nabla r_K\rvert^{p-1}
		\,\Haus^{n-1}\!\restriction_{\partial K}
		\right)
		=
		S_p(K,\cdot).
		\]
		Therefore, the measure appearing in
		\eqref{eq:convexification-definition} is
		\[
		\lvert g'(s)\rvert^{p-1}s^{n-1}\,dS_p(K,\cdot).
		\]
		Since
		\[
		dS_p(cK,\cdot)=c^{n-p}\,dS_p(K,\cdot),
		\]
		the choice of \(c=c(t)\) in \eqref{eq:level-dilate} gives precisely
		this measure. Uniqueness for the even \(L_p\)-Minkowski problem
		implies
		\[
		\convlev{f}{t}=c(t)K.
		\]

		Both \(\Phi_p\) and \(\pcap_p\) are homogeneous of degree \(n-p\).
		Hence, for either
		\(
		\Psi=\Phi_p
		\ \text{or}\
		\Psi=\pcap_p,
		\)
		equation \eqref{eq:level-dilate} and the substitution \(t=g(s)\)
		yield
		\begin{align}
			\int_0^\infty
			\Psi\bigl(\convlev{f}{t}\bigr)\,dt
			=
			A_{g,p}\,\Psi(K),
			 \ \ \text{with} \ \ 
			A_{g,p}
			=
			\int_1^2
			\lvert g'(s)\rvert^p s^{n-1}\,ds
			>0.\label{eq:common-factor}
		\end{align}
		Since \(0\leq f\leq1\), the relevant level set is empty for \(t>1\),
		and the endpoint levels have measure zero.
		Therefore, because \(A_{g,p}>0\), combining
		\eqref{eq:common-factor} with
		\[
		\Phi_p(K)>\gammaconst\,\pcap_p(K)
		\]
		proves \eqref{eq:analytic-reverse}.
	\end{proof}
	
	\begin{proof}[Proof of (ii) in Theorem \ref{thm:main}]

	Part~\textup{(ii)} follows by applying
	Proposition~\ref{prop:analytic-counterexample} to a convex body
	constructed in Proposition~\ref{prop:symmetric-counterexample}.
	This completes the proof of Theorem~\ref{thm:main}.
		\end{proof}
	
	\section*{ Appendix: Verification of the special case \((n,p)=(3,2)\)}
	\label{sec:explicit}
	
	The preceding construction is obtained by a perturbative argument.
	We conclude by exhibiting an explicit convex body with a quantitative
	gap, thereby providing an independent verification of the
	normalizations in the special case \((n,p)=(3,2)\). For an even
	integer \(q\geq2\), define
	\[
	B_q^3
	:=
	\left\{
	x\in\R^3:
	x_1^q+x_2^q+x_3^q\leq1
	\right\}.
	\]
	
	\begin{proposition}\label{prop:explicit}
		For an even integer $q\geq2$,
		\begin{equation}\label{eq:exact-S2}
			\Phi_2(B_q^3)=S_2(B_q^3)
			=
			\frac{24\pi(q-1)}
			{q(q+1)\sin(\pi/q)}.
		\end{equation}
		In particular,
		\[
		\Phi_2(B_{24}^3)>\pcap_2(B_{24}^3).
		\]
	\end{proposition}
	
	\begin{proof}
		For a smooth convex body \(K\in\mathcal K_o^3\), define the symmetric
		matrix
		\[
		M_K
		=
		\int_{\partial K}
		\frac{\nu_K(x)\otimes\nu_K(x)}
		{x\cdot\nu_K(x)}
		\,d\Haus^2(x).
		\]
		Note that for
		\(a,b\in\R^3\), the tensor product \(a\otimes b\) is the matrix defined
		by
		\[
		(a\otimes b)_{ij}:=a_i b_j,
		\ \ 1\leq i,j\leq3.
		\]
		Thus,
		\[
		(M_K)_{ij}
		=
		\int_{\partial K}
		\frac{\nu_{K,i}(x)\nu_{K,j}(x)}
		{x\cdot\nu_K(x)}
		\,d\Haus^2(x).
		\]
		Since \(K\in\mathcal K_o^3\), we have
		\(x\cdot\nu_K(x)>0\) on \(\partial K\). Hence, the matrix \(M_K\) is
		well defined and positive semidefinite.
	Equation \eqref{eq:Vp} yields
	\[
	V_2(K,\bar u)
	=
	\frac{1}{12}\,u^{\mathsf T}M_Ku.
	\]
	The body \(B_q^3\) is invariant under arbitrary sign changes and
	permutations of its coordinates. These symmetries imply that
	\[
	M_{B_q^3}=m_q I
	\]
	for some \(m_q>0\). Consequently,
	\[
	V_2(B_q^3,\bar u)=\frac{m_q}{12}
	\]
	is independent of \(u\in\Sph^2\). The normalization formulas
	\eqref{eq:Sp-average} and \eqref{eq:Phi} thus give
	\[
	\Phi_2(B_q^3)=S_2(B_q^3).
	\]
		
	Define the vector field
	\[
	X(x)
	=
	\bigl(x_1^{q-1},x_2^{q-1},x_3^{q-1}\bigr).
	\]
	For \(x\in\partial B_q^3\), we have
	\[
	\nu(x)=\frac{X(x)}{\lvert X(x)\rvert},
	\ \ 
	x\cdot X(x)=1.
	\]
	It follows that
	\[
	\frac{1}{x\cdot\nu(x)}
	=
	\lvert X(x)\rvert
	=
	X(x)\cdot\nu(x).
	\]
	Therefore, by the divergence theorem and coordinate symmetry,
	\begin{align}
		S_2(B_q^3)
		&=
		\int_{\partial B_q^3}
		X\cdot\nu\,d\Haus^2
		=
		\int_{B_q^3}\operatorname{div}X\,dx
		\nonumber\\
		&=
		3(q-1)
		\int_{B_q^3}
		\lvert x_1\rvert^{q-2}\,dx.
		\label{eq:divergence}
	\end{align}
	To compute the remaining integral, decompose \(B_q^3\) into its eight
	orthants and introduce the variables
	\[
	y_i=\lvert x_i\rvert^q,
	\ \ i=1,2,3.
	\]
	The Dirichlet simplex integral then gives
	\[
	\int_{B_q^3}
	\lvert x_1\rvert^{q-2}\,dx
	=
	\frac{8}{q^3}
	\frac{
		\Gamma(1-1/q)\Gamma(1/q)^2
	}{
		\Gamma(2+1/q)
	}.
	\]	Inserting this expression into \eqref{eq:divergence} and applying
	Euler's reflection formula
	\[
	\Gamma(z)\Gamma(1-z)
	=
	\frac{\pi}{\sin(\pi z)}
	\]
	yields \eqref{eq:exact-S2}.
		
	For \(q=24\), it follows that
	\[
	\Phi_2(B_{24}^3)
	=
	\frac{23\pi}{25\sin(\pi/24)}
	>
	\frac{552}{25}
	>
	22,
	\]
	where the first strict inequality follows from
	\(\sin x<x\) for \(x>0\). On the other hand,
	\[
	B_{24}^3
	\subset
	[-1,1]^3
	\subset
	\sqrt{3}\,B^3.
	\]
	By monotonicity and homogeneity of the \(2\)-capacity, together with
	\(\pcap_2(B^3)=4\pi\), we obtain
	\[
	\pcap_2(B_{24}^3)
	\leq
	\pcap_2(\sqrt{3}\,B^3)
	=
	4\pi\sqrt{3}
	<
	22.
	\]
	The final strict inequality follows from
	\[
	\pi<\frac{22}{7},
	\ \
	\sqrt{3}<\frac{7}{4}.
	\]	
	
	\end{proof}

	Applying the construction from
	Proposition~\ref{prop:analytic-counterexample} to $K=B_{24}^3$ gives a
	fully explicit smooth counterexample to the analytic inequality in $\R^3$
	at $p=2$.

\end{document}